\documentclass[smallcondensed]{svjour3}
\usepackage{amsmath}
\usepackage{mathptmx}
\usepackage{amsfonts}
\usepackage{amssymb}
\usepackage{comment}
\usepackage{graphicx}
\usepackage{mathptmx}
\usepackage{marvosym}
\usepackage{url}

\usepackage[ruled,vlined,linesnumbered]{algorithm2e}
\usepackage{pstricks-add,pst-slpe}
\usepackage{pstricks}
\usepackage{pst-bar}
\usepackage{pstricks-add}

\newcommand{\yy}{\vec{y}}
\newcommand{\xx}{\vec{x}}

\def \va     {{\vec \alpha}}
\def \vb    {{\vec \beta }}
\def \vr    {{\vec \gamma }}
\newcommand{\ZZ}{\mathbb{Z}_+}

\newcommand{\QQ}{\mathbb{Q}}
\newcommand{\RR}{\mathbb{R}}

\def \algrelate {{\tt MonomialRelation}}
\def \algexa {{\tt EXACTG}}

\def \algSOS{{\tt SOS}}

\def \algf{{\tt QuickSOS}}
\def \algpca{{\tt PCAG}}
\def \aisat{{\tt QuickSOS}}
\def \sostool{{\tt SOSTOOLS}}
\def \yalmip{{\tt YALMIP}}
\def \sosopt{{\tt SOSOPT}}

\def \DEF {{\stackrel{\text{\tiny def}}{=}}}

\def \VV {{\mathcal V}}

\def \CONV {{\rm conv}}

\def \NEW {{\tt N}}
\def \PROJ {{\tt Proj}}
\def \SUPPORT {{\tt S}}

\def \GSUP {{\mathbf Q}}

\def \GSUPZ{{{ \mathbf Q}^0}}
\def \PSUP{{{\mathbf P}}}
\def \PSUPE{{{\mathbf P}^e}}

\def \MAP{{\varphi}}
\def \RELAT {{\psi}_{\GSUP}}
\def \SOSS{{SOSS}}

\newcommand*\arc[3]{%
  \ncline{#1}{#2}
 }

\newtheorem{nota}{Notation}

\title{Smaller SDP for SOS Decomposition}
\author{Liyun Dai \and Bican Xia }
\institute{Liyun Dai(\Letter)
	\at  LMAM \& School of Mathematical Sciences, Peking University \\
	Beijing International Center for Mathematical Research, Peking University
	\\\email{dailiyun@pku.edu.cn}
	\and
	Bican Xia
	\at  LMAM \& School of Mathematical Sciences, Peking University
	\\\email{xbc@math.pku.edu.cn}
}
\begin{document}  \mag=1100
\label{firstpage}
\date{}
\maketitle

\begin{abstract}

A popular numerical method to compute SOS (sum of squares of polynomials) decompositions for polynomials is to transform the problem into semi-definite programming (SDP) problems and then solve them by SDP solvers. In this paper, we focus on reducing the sizes of inputs to SDP solvers to improve the efficiency and reliability of those SDP based methods. 
Two types of polynomials, convex cover polynomials and split polynomials, are defined. A convex cover polynomial or a split polynomial can be decomposed into several smaller sub-polynomials such that the original polynomial is SOS if and only if the sub-polynomials are all SOS. Thus the original SOS problem can be decomposed equivalently into smaller sub-problems. It is proved that convex cover polynomials are split polynomials and it is quite possible that sparse polynomials with many variables are split polynomials, which can be efficiently detected in practice. Some necessary conditions for polynomials to be SOS are also given, which can help refute quickly those polynomials which have no SOS representations so that SDP solvers are not called in this case. All the new results lead to a new SDP based method to compute SOS decompositions, which improves this kind of methods by passing smaller inputs to SDP solvers in some cases. Experiments show that the number of monomials obtained by our program is often smaller than that by other SDP based software, especially for polynomials with many variables and high degrees. Numerical results on various tests are reported to show the performance of our program.

\end{abstract}

\section{Introduction}
\label{sec:intro}

Since Hilbert's seventeenth problem was raised in 1900, there has been a lot of work on SOS (sums of squares of polynomials) decomposition. To name a few, please see for instance  \cite{artin27,reznick1995uniform,shor1998nondifferentiable,nesterov2000squared,lasserre2001global,Parrilo03,parrilo2003minimizing,blekherman2006there}.

From an algorithmic view of point, writing a multivariate polynomial as an SOS to prove it is non-negative is a crucial part of many applications \cite{vand96,lasserre2001global,Parrilo03,kim05,sch05} though the number of non-negative polynomials are much more than the number of sum of squares polynomials \cite{blekherman2006there}. Numerical algorithms for SOS decompositions can handle big scale problems and can be used to get exact results \cite{kaltofen2008exact}. One main numerical method to solve SOS decomposition problem is to convert it to SDP problem. Actually, there exist some well-known free available SOS solvers which are based on SDP solvers \cite{PPP02,lofberg2004yalmip,seiler2013sosopt}.

Obviously, improving SDP solvers' efficiency can improve the efficiency of SDP based SOS solvers. For related work on improving SDP solvers' efficiency, please see for example \cite{nesterov1994interior,mont97,Stur98,wolkowicz2000handbook,renegar2001mathematical,todd2001semidefinite,mats12}.
It is known that, in the worst case, the size of corresponding SDP problem is $O({d+n \choose d})$ which is polynomial in both $n$ (the number of variables) and $d$  (the degree of given polynomial), if the other one is fixed. In practice, the size of corresponding SDP can be much smaller than $O({d+n \choose d})$ \cite{lofberg2009pre,waki2010facial,seiler2013simplification}.
Although the complexity of SDP is polynomial in $d$ and $n$, the actual complexity of SDP based SOS solvers are very high since the size of corresponding matrices of SDP is very large when the given polynomial has many variables and high degree.
Moreover, the results of existing SDP solvers may be not reliable for large problems \cite{DGXZ13}. In other words, it is important to reduce the size of corresponding SDP problem so as to improve both the efficiency and reliability of SDP based SOS solvers.

In many practical situations, we do not know more properties of the given polynomial except that the polynomial is sparse, {\it i.e.}, the number of monomials is much smaller than ${d+n \choose d}$. So how to take use of the sparsity to reduce the corresponding size of SDP is a key part to improve the efficiency of solving SOS decomposition problem. 
For related work on employing sparsity, see for instance \cite{renz78,kojima05,waki07}.
For SOS decomposition of a polynomial on an algebraic variety, a method which may yield smaller SDP is proposed in \cite{permenter2012selecting} by combining Gr\"obner basis techniques with Newton polytope reduction.

In this paper, we focus on reducing the sizes of inputs to SDP solvers to improve the efficiency and reliability of those SDP based methods. Two types of polynomials, convex cover polynomials and split polynomials, are defined. A convex cover polynomial or a split polynomial can be decomposed into several smaller sub-polynomials such that the original polynomial is SOS if and only if the sub-polynomials are all SOS. Thus the original SOS problem can be decomposed equivalently into smaller sub-problems. It is proved that convex cover polynomials are split polynomials and it is quite possible that sparse polynomials with many variables are split polynomials, which can be efficiently detected in practice. Some necessary conditions for polynomials to be SOS are also given, which can help refute quickly those polynomials which have no SOS representations so that SDP solvers are not called in this case. For example, the well-known Motzkin polynomial \cite{motzkin1967arithmetic} and Choi-Lam example \cite{choilam77} do not pass the check of the necessary conditions. All the new results lead to a new SDP based method to compute SOS decompositions, which improves this kind of methods by passing smaller inputs to SDP solvers in some cases. Experiments show that the number of monomials obtained by our program is often smaller than that by other SDP based software, especially for polynomials with many variables and high degrees. Numerical results on various tests are reported to show the performance of our program.

The rest part of this paper is organized as follows.  Section \ref{sec:pro} describes some notations and existing results on SOS, which will be used in this paper. Convex cover polynomial is defined and its property is proved based on the convex property of corresponding Newton polytopes in Section \ref{sec:work}. Split polynomial is defined and its property is proved based on monomial relation analysis in Section \ref{sec:mon}. Moreover, the relationship between the two types of polynomials is given also in Section \ref{sec:mon}. A new algorithm for SOS decomposition based on those new results is presented in Section \ref{sec:imp}. We report some experimental data of our program with comparison to other SDP based tools in Section \ref{sec:exp}.

\section{Preliminary}\label{sec:pro}

The symbols $\mathbb{Z}, \ZZ, \QQ$ and $\RR$ denote the set of integers, natural numbers, rational numbers and real numbers, respectively. If not specified, ``polynomials" in this paper are polynomials with real coefficients and are often denoted by $p,q,f,g$, etc.. By ``vectors" we mean vectors in $\ZZ^n$ (or $\RR^n$) which are denoted by $\va, \vb, \vr$, etc.. We use $\xx, \yy$ denote the variable vectors $(x_1,\ldots,x_n), (y_1,\ldots,y_n)$, respectively. A hyperplane in $\RR^n$ is denoted by $\pi(\xx)=0$.

Consider a polynomial
\begin{equation}
	p(\xx)=\sum_{\va\in {\PSUP}}c_{\va}\xx^{\va}
	\label{eq:poly}
\end{equation}
in the variable vector $\xx\in \RR^n$ with a {\em support} ${\PSUP}\subseteq
\ZZ^n$, where $\ZZ\DEF\{x\in \mathbb{Z},x\ge0\}$  and real coefficients $c_{\va}\neq 0\ (\va\in {\PSUP})$. Denote by
$\SUPPORT (p)$ the support of a polynomial $p$. For example, if $p=1+x_1^2+x_2^3$, then $n=2,\SUPPORT (p)=\{(0,0),(2,0),(0,3)\}$.
When $p=0,$ define $ \SUPPORT(p)=\emptyset$.

For any $T\subseteq \RR^n$ and $k\in\RR$, denote by $kT$ the set $\{k\va \mid  \va\in T\}$, where   $k(a_1,\dots,a_n)=(ka_1,\dots,ka_n)$, and
by {\rm conv}$(T)$ the convex hull of  $T$.
Let $\PSUPE$ be the set of
$\va\in {\PSUP}$ whose coordinates $\va_k\ (k=1,2,\dots,n)$
are all even non-negative integers, {\it i.e.}, ${\PSUPE}={\PSUP}\cap(2\ZZ^n)$.


Obviously, $p(\xx)$ can be represented in terms of a sum of squares of
polynomials or in short, $p$ is SOS, if and only if there exist polynomials $q_1(\xx),\dots,q_s(\xx)\in \RR[\xx]$ such that
\begin{equation}
	\label{eq:sum}
	p(\xx)=\sum_{i=1}^sq_i(\xx)^2.
\end{equation}
To find both $s$ and polynomials $q_1(\xx),\ldots,q_s(\xx)$, it is necessary to estimate and decide the supports
of unknown polynomials $q_i(\xx)  (i=1,\dots,s)$. 
Let ${\GSUP}_i$ be an unknown support of the polynomial
$q_i(\xx)\ (i=1,\dots,s)$. Then each polynomial $q_i(\xx)$ is
represented as

\begin{equation}
  \label{eq:sing}
	q_i(\xx)=\sum_{\va\in {\GSUP}_i}c_{(i,\va)}\xx^{\va}
\end{equation}
with  nonzero  coefficients $c_{(i,\va)}\ (\va\in
{\GSUP}_i,\ i=1,\dots,s)$.

Suppose $p(\xx)$ is of the form (\ref{eq:sum}), then 
${\PSUP}\subseteq {\rm conv}(\PSUPE)$ \cite{renz78}. The
following relation is also known by \cite{renz78}:
\begin{equation}
\label{eq:sup}
	\left\{ \va\in \ZZ^n :\va\in {\GSUP}_i \mbox{ and }
	c_{(i,\va)}\neq 0 \mbox{ for some } i\in \{  1,2,\dots,s \} \right\}  \subseteq \frac{1}{2} {\rm conv}(\PSUPE).
\end{equation}

Hence we can confine effective supports of unknown polynomials
$q_1(\xx),\ldots,q_s(\xx)$ to subsets of
\begin{equation}
	\label{eq:subg}
	{\GSUP}^0=\left(\frac{1}{2}{\rm conv}({\PSUPE})\right)\cap \ZZ^n.
\end{equation}

\begin{definition}\label{def:soss}
For a polynomial $p$, a set $\GSUP\subseteq \ZZ^n$ is said to satisfy the relation $\SOSS(p,\GSUP)$ (SOSS stands for SOS support) if $$ p \mbox{ is SOS }\implies \mbox{ there exist } q_i (i=1,\ldots, s) \mbox{ such that }  p=\sum_{i=1}^sq_i^2 \mbox{ and } \SUPPORT(q_i)\subseteq \GSUP.  $$
\end{definition}

For every given $p$, the problem is how to find a small $\GSUP$ such that  $\SOSS(p,{\GSUP})$ holds, {\it i.e.},  prune more unnecessary monomials
from the decomposition. In general, one can start from  a coarse  ${\GSUP}$, keep eliminating elements of ${\GSUP}$ which does not satisfy some conditions, and finally  obtain a smaller ${\GSUP}$.   Obviously, ${\GSUP}^0$ of (\ref{eq:subg}) satisfies $\SOSS(p,{\GSUP}^0)$ for every given $p$.
If $q_i$ satisfies (\ref{eq:sum}), the relation $\SOSS(p,\cup_i\SUPPORT(q_i))$ holds.

There are two possible approaches to construct $q_1,\cdots,q_s$.

One approach assumes polynomials $q_1,\cdots,q_s$ do not share common support. Then each polynomial $q_i(\xx)$ is represented as Eq. (\ref{eq:sing}).
Unfortunately, it is difficult to find exact ${\GSUP}_i$ if we do not know more information of $p$. But when $p$ is correlatively sparse,
a correlative sparsity pattern graph is defined in \cite{waki07}  to find a certain sparse structure  in $p$. And this structure  can be used to
decide different relaxed ${\GSUP}_i$. Theoretically, the relaxations in \cite{waki07} are not guaranteed to generate lower bounds with the same quality as those generated by the original SOS representation.

The other approach assumes that all polynomials
$q_1(\xx),\dots,q_s(\xx)$ share a common unknown support
${\GSUP}\subseteq {\GSUP}^0$ and   each
polynomial $q_i(\xx)$ is represented as
\begin{equation}
	\label{eq:rep}
	q_i(\xx)=\sum_{\va\in {\GSUP}}c_{(i,\va)}\xx^{\va}.
\end{equation}

Then Eq. (\ref{eq:sum}) is equivalent to the existence of a
positive semi-definite matrix $M$ such that
\begin{equation}
	\label{eq:SOS1}
	p(\xx)={\GSUP}^T(\xx)M{\GSUP}(\xx),
\end{equation}
where ${\GSUP}(\xx)$ is a vector of monomials corresponding to the support $\GSUP$. So in the view of practical computing, finding the SOS representation is equivalent to solving the feasibility problem of (\ref{eq:SOS1}). 
Thus, the original problem can be solved by SDP solvers. This approach was presented in  \cite{PPP02,kim05,lofberg2004yalmip,seiler2013sosopt}. There are close connections between SOS polynomials and positive semi-definite matrices \cite{choi1995sums,powers1998algorithm,Parrilo00,lasserre2001global}.

\begin{nota}
\label{not:1}
We denote  by $\algSOS(p,{\GSUP})$ an algorithm of finding positive semi-definite matrix $M$ with $\GSUP$ under constraints (\ref{eq:SOS1}).
\end{nota}

Let us give a  rough complexity analysis of $\algSOS(p,{\GSUP})$.
Let $n=\#(\GSUP)$, the number  of elements contained in $\GSUP$. Then the size of matrix $M$ in (\ref{eq:SOS1}) is $n\times n$.
Let $m$ be  the number of different elements occurring in $\GSUP\GSUP^T$. It is easy to know $n\le m \le n^2$. Suppose $m=O(n^c),c\in [1,2]$ and we use {\em interior point method} in $\algSOS(p,{\GSUP})$, which is a main method for solving SDP. Then the algorithm will repeatedly solve {\em least squares} problems with $m$ linear constraints and $\frac{(n+1)n}{2}$ unknowns.
Suppose that the least squares procedure is called  $k$ times. Then, the total complexity is $O(k  n^{2+2c})$. So, if $n$ becomes $2n$, the time consumed  will increase by at least $16$ times.
So reducing $\GSUP$'s size is a key point to improve such algorithms.

\section{Convex cover polynomial}
\label{sec:work}

We give a short description of Newton polytope in Section \ref{sec:con}. In Section \ref{sec:cpro}, we first prove a necessary condition (Theorem \ref{the:1}) for a polynomial to be SOS based on the properties of Newton polytope. Then a new concept, {\em convex cover polynomial} (Definition \ref{def:cv}), is introduced, which leads to the main result (Theorem \ref{cor:2}) of this section, that is, a convex cover polynomial is SOS if and only if some smaller polynomials are SOS.

\subsection{Newton polytope}
\label{sec:con}

\textit{Newton polytope} is a classic tool. We only introduce some necessary notations here. For
formal definitions of the concepts, please see for example \cite{sturm98}. A {\em polytope} is a subset of $\RR^n$ that is the convex hull of a finite
set of points. A simple example is the convex hull of
$\{(0,0,0),(0,1,0),(0,0,1),(0,1,1),(1,0,0),
(1,1,0),(1,0,1)(1,1,1)\}$
in $\RR^3$; this is the regular $3$-cube. A $d$-dimensional polytope
has many {\em faces}, which are again polytopes of various dimensions
from $0$ to $d-1$. The $0$-dimensional faces are called {\em vertices},
the $1$-dimensional faces are called {\em edges}, and the
$(d-1)$-dimensional faces are called {\em facets}. For instance, the cube
has $8$ vertices, $12$ edges, and $6$ facets. If $d=2$ then the edges
coincide with the facets. A $2$-dimensional polytope is called a
\textit{polygon}.

For a given polynomial $p$, each term $\xx^{\va}=x_1^{a_1}\cdots x_n^{a_n}$ appearing in $p$
corresponds to an integer lattice point $(a_1,\ldots,a_n)$ in
$\RR^n$. The convex hull of all these points (called the {\em support} of $p$) is defined as
\textit{Newton polytope} of $p$ and is denoted by
$$\NEW(p)\DEF \CONV( \SUPPORT (p)).$$

\begin{definition}
  For a polynomial $p=\sum_{\va}c_{\va}\xx^{\va}$ and a set $T\subseteq \RR^n$, denote by
  $\PROJ (p,T)$ the polynomial obtained by deleting the terms $c_{\va}\xx^{\va}$ of $p$ where
  $\va \not \in (T\cap \ZZ^n)$.
\end{definition}

\begin{example}
  \label{ex:1}

  Let $p=2x_1^{4}+4x_2^{4}-3 x_3^{2}+1$ and
  $T=\{(0,0,0),(1,0,0),(4,0,0)\}$, then $\PROJ (p,T)$ $=2x_1^4+1$.

\end{example}

\subsection{Convex cover polynomial}
\label{sec:cpro}

We guess that Theorem \ref{the:1} in this section should be a known result. However, we do not find a proof in the literature. So, we prove it here. 
Since the results of the following Lemma \ref{lem:pre} are either obvious or known, we omit the proofs.

\begin{lemma}
\label{lem:pre}
  \begin{itemize}
  \item   For any two polynomials $f,g$, two real numbers $k_1,k_2$  and any
    $T\subseteq \ZZ^n$, $$\PROJ (k_1f+k_2g,T)=k_1\PROJ (f,T)+k_2\PROJ (g,T).$$
  \item   For any $T\subseteq \ZZ^n$ and any $k\in
    \RR\setminus\{0\}$, we have $k(\frac{1}{k}T\cap\ZZ^n)\subseteq T .$

  \item   Suppose  $N$ is an $n$-dimensional polytope. For any face $F$ of $N$,
  there is an $(n-1)$-dimensional  hyperplane $\pi(\yy)=0$ such that
  $\pi(\va)=0$ for any $ \va\in F$ and $\pi(\vb)>0 $ for any $ \vb \in
  N\setminus F$.
  \item   Suppose  $\pi(\yy)=0$ is a hyperplane and
  $F \subseteq \ZZ^n \cap (\pi(\yy)=0)$. For any polynomial
  $p=\sum_{\va}c_{\va}\xx^{\va}$ in $n$ variables, we have $$\SUPPORT(\PROJ
  (p,F)) \subseteq \SUPPORT(\PROJ(p, \SUPPORT(p)\cap (\pi(\yy)=0))).$$
   \item   If $f,g$ are two polynomials and $\SUPPORT(f)\cap\SUPPORT(g)=\emptyset
  $, then $\SUPPORT(f+g)=\SUPPORT(f)\cup \SUPPORT(g)$.
  \item   Let $T_1=\SUPPORT(f)$ and $T_2=\SUPPORT(g)$ for two polynomials $f$ and $g$.
  Then $\SUPPORT(fg)\subseteq T_1+T_2$, where $T_1+T_2$ is the {\em Minkowski
    sum} of $T_1$ and $T_2$.

  \end{itemize}

\end{lemma}

\begin{lemma}
  \label{le:5}
  Suppose	$\pi(\yy)=0$ is a hyperplane, $T\subseteq \ZZ^n$   and
  $f,g$  are two $n$-variate
  polynomials. Let $T_1=\SUPPORT (f),T_2=\SUPPORT (g)$.  If  $ T\subseteq
  \{\yy \mid \pi(\yy)=0\}, 2T_1\subseteq \{\yy \mid \pi(\yy)\geq 0\}$ and
  $2T_2\subseteq \{\yy \mid \pi(\yy)>0\}$, then $\PROJ (fg,T)=0$.

\end{lemma}

\begin{proof}
  By Lemma \ref{lem:pre}, $\SUPPORT(fg)\subseteq T_1+T_2$. By the definition of Minkowski sum,
  for  any $\va\in T_1+T_2$  there exist
  $\va_1\in T_1 , \va_2 \in T_2$ such that
  $\va=\va_1+\va_2$. Because $ \pi(2\va_1)\geq 0$ and $\pi(2\va_2)>0$,
  $\pi(\va)=\pi(\va_1+\va_2)=\frac{1}{2}(\pi(2\va_1)+\pi(2\va_2) )>0$.
  So $T_1+T_2\subseteq  \{\yy \mid
  \pi(\yy)>0\}$. Thus, $\SUPPORT(fg)\cap (\pi(\yy)=0)=\emptyset$ which implies $\PROJ(fg,T))=0$ by Lemma
  \ref{lem:pre} and $T\subseteq (\pi(\yy)=0)$. \qed
\end{proof}

\begin{lemma}
  \label{le:10}

  Suppose $p=\sum_{i=1}^sq_i^2$ and $F$ is a face of
  $\NEW(p)$. Let  $F_z=F\cap\ZZ^n,F_{\frac{z}{2}}=\frac{1}{2}F\cap \ZZ^n,
  q_i'=\PROJ(q_i,F_{\frac{z}{2}}),$ $q_i''=q_i-q_i',T_i'=\SUPPORT(q_i')$ and $T_i''=\SUPPORT(q_i'')$,
  then there is a hyperplane $\pi(\yy)=0$ such that
  \begin{enumerate}
  \item[(1)] $F\subseteq \{\yy \mid \pi(\yy)=0\}$,
  \item[(2)] $2T_i'\subseteq \{\yy \mid \pi(\yy)=0\}$, and
  \item[(3)] $ 2T_i''\subseteq	\{ \yy \mid \pi(\yy)>0\}$.
  \end{enumerate}
\end{lemma}

\begin{proof}

  By Lemma \ref{lem:pre}, there is a hyperplane $\pi(\yy)=0$ such that $\forall
  \va\in F, \pi(\va)=0$ and $\forall \va\in \NEW(p)\setminus
  F, \pi(\va)>0$. We  prove that   $\pi$   is a  hyperplane which
  satisfies the requirement. First, because $T_i'\subseteq F_{\frac{z}{2}}$, by
  Lemma  \ref{lem:pre}, $2T_i'\subseteq 2F_{\frac{z}{2}}\subseteq F$ and thus $2T_i'\subseteq \{\yy \mid \pi(\yy)=0\}$.

  Second, it is obvious that   $T_i ''\cap F_{\frac{z}{2}}=\emptyset$,  $T_i'\cap T_i''=\emptyset$ and $T_i'\cup T_i''= T_i$ where $T_i=\SUPPORT(q_i)$. By Equation (\ref{eq:sup}), we have $T_i
  \subseteq \frac{1}{2}\NEW(p)$ and $ 2T_i\subseteq \NEW(p)$. Thus  $ 2T_i''\subseteq 2T_i\subseteq \NEW(p) \subseteq \{\yy \mid \pi(\yy) \geq 0 \}$. If there is an $\va\in T_i''$ such that
  $\pi(2\va)=0$, then  $\va\in
  F_{\frac{z}{2}}$, which contradicts with $ T_i ''\cap
  F_{\frac{z}{2}}=\emptyset$. Therefore, $ 2T_i''\subseteq
  \{ \yy \mid \pi(\yy)>0\}$. \qed
\end{proof}

Using the above lemmas, we prove Theorem \ref{the:1} now.

\begin{theorem}
  \label{the:1}
  If $p$ is SOS, then $\PROJ (p,F)$ is SOS for every face $F$	of $\NEW(p)$.
\end{theorem}
\begin{proof}
  Suppose $p=\sum_{i=1}^sq_i^2$ and $F$ is a face of
  $\NEW(p)$. Let  $F_z=F\cap\ZZ^n,
  q_i'=\PROJ(q_i,\frac{1}{2}F_z)$ and $q_i''=q_i-q_i'$. Then
  $p=\sum_{i=1}^s(q_i'+q_i'')^2=\sum_{i=1}^sq_i'^2+2\sum_{i=1}^sq_i'q_i''+\sum_{i=1}^sq_i''^2$. By Lemma \ref{lem:pre},
  $\PROJ (p,F_z)=\sum_{i=1}^s\PROJ (q_i'^2,F_z)+2\sum_{i=1}^s\PROJ (q_i'q_i'',
  F_z)+\sum_{i=1}^s\PROJ (q_i''^2,F_z)$. By Lemma \ref{le:10}, there is a hyperplane $\pi(\yy)=0$ such that
  (1) $\forall \va\in F, \pi(\va)=0$; (2) $\forall \va\in
  \NEW(p)\setminus F, \pi(\va)>0$; (3) for any $q_i',2\SUPPORT (q_i')\subseteq
  \{\yy \mid \pi(\yy)=0\}$; and (4) $2\SUPPORT (q_i'')\subseteq \{\yy \mid \pi(\yy)>0\}$. By Lemma \ref{le:5},
  $\PROJ (q_i'q_i'',F_z)=0$ and $\PROJ (q_i''^2,F_z)=0$. Therefore,  the intersect between  support of $2\sum_{i=1}^sq_i'q_i''+\sum_{i=1}^sq_i''^2$
  and $F_z$ is an emptyset, \it{i.e.},
  $\PROJ (p,F)=\PROJ (p,F_z)=\sum_{i=1}^s\PROJ (q_i'^2,F_z)$   $=\sum_{i=1}^s(q_i')^2$. The last equality holds because
  $\SUPPORT (q_i'^2)\subseteq F_z$. \qed
\end{proof}

\begin{remark}
Theorem \ref{the:1} is strongly related to Theorem 3.6 of \cite{reznick1989forms}, which states that   if $p$ is positive semidefinite, then $\PROJ (p,F)$ is positive semidefinite  for every face $F$	of $\NEW(p)$.

\end{remark}

Theorem \ref{the:1} proposes a necessary condition for a polynomial to be SOS.

\begin{example}
  \label{ex:3}
  $p=x_1^{4}+x_2^{4}+x_3^{4}-1.$
\end{example}

Obviously, the polynomial in Example \ref{ex:3} is not SOS ({\it e.g.}, $p(0,0,0)=-1$).
By Theorem \ref{the:1}, one necessary condition for $p$ to be SOS is that $\PROJ (p,\{(0,0,0)\})=-1$ should be SOS which can be efficiently checked. On the other hand, if we use Newton polytope based method to construct $\GSUP$ in (\ref{eq:SOS1}), the size of $\GSUP$ is ${3+2 \choose 2}=10$. The number of constraints is ${3+4 \choose 4}=35$.

\begin{definition}[Convex cover polynomial]
  \label{def:cv}
  A polynomial $p$ is said to be a {\em convex cover  polynomial} if there exist some pairwise disjoint
  faces $F_i (i=1,\dots, u)$  of $\NEW(p)$ such that 
  $\SUPPORT(p)\subseteq \cup_{i=1}^uF_i $.
\end{definition}

It is easy to get the following proposition by the definition of convex cover  polynomial.  
\begin{proposition}
  \label{lem:int}
  The support of a convex cover  polynomial does not intersect the interior of its Newton polytope.
\end{proposition}

The following theorem is a direct corollary of Theorem \ref{the:1}.

\begin{theorem}
  \label{cor:2}
  Suppose $p$ is a convex cover polynomial and $F_i (i=1,\dots u)$ are the faces satisfying the condition of Definition \ref{def:cv}.
  Let $p_i=\PROJ (p,F_i) (i=1,\dots u)$. Then $p$ is SOS if and only if $p_i$ is SOS for $i=1,\dots u$. 
\end{theorem}

We use the following example to demonstrate the benefit of Theorem \ref{cor:2}.
\begin{example}
  \label{ex:2}
  $p=x_1^6+x_2^6+x_1^4-2x_1^2x_2^2+x_2^4.$
\end{example}

\begin{figure}[h]
  \begin{center}
    \psset{xunit=0.5,yunit=0.5}
    \begin{pspicture}(-2,-2)(7.3,7.3)
        \psaxes{->}(0,0)(-0.14,-0.14)(7.3,7.3)
        \psaxes[
        dy=1,
        Dy=1,
        labels=y,
        ticks=y,
        ](0,0)(-0.3,-0.3)(7.3,7.3)
      {\psset{fillstyle=ccslope,slopebegin=cyan!40,slopeend=yellow}
        \cnodeput(4,0){K}{\tiny {\strut\boldmath$x_1^4$}}
        \cnodeput(6,0){F}{\tiny {\strut\boldmath$x_1^6$}}
        \cnodeput(0,6){D}{\tiny {\strut\boldmath$x_2^6$}}
        \cnodeput(0,4){H}{\tiny {\strut\boldmath$x_2^4$}}
        \cnodeput(2,2){A}{\tiny {\strut\boldmath$x_1^2x_2^2$}}
      }

      \arc{K}{F}{120}
      \arc{F}{D}{650}
      \arc{D}{H}{780}
      \arc{H}{A}{490}
      \arc{A}{K}{600}

    \end{pspicture}

    \caption{Newton polytope of  Example \ref{ex:2}. \label{fig:con}}
  \end{center}
\end{figure}
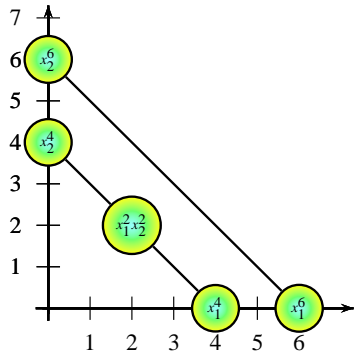

For Example \ref{ex:2},  $\SUPPORT(p)=\{(6,0),(0,6),(4,0),(2,2),(0,4)\}$.
Let $F_1=\{(6,0)\}, F_2=\{(0,6)\},$ $ F_3=\CONV(\{(4,0),(2,2),(0,4) \} )$ (shown in Fig. \ref{fig:con}) be three faces of $\NEW(p)$.
Because $F_1, F_2, F_3$ satisfy the condition of Definition \ref{def:cv}, $p$ is a convex cover polynomial. Let $p_i=
\PROJ(p,F_i)$ for $i=1,2,3$. Then, by Theorem \ref{cor:2}, proving $p$ is SOS is equivalent to proving $p_i$
is SOS for $i=1,2,3$. Therefore, the original problem is divided into three simpler sub-problems.
When using Newton polytope based method to prove $p$ is SOS, the size of $\GSUP$ is $7$ and the number of
constraints is $18$. However, for $p_1,p_2,p_3$, the corresponding data are $(1,1)$, $(1,1)$ and $(3,5)$, respectively.

Dividing the original
problem into simpler sub-problems can  improve  not only the efficiency but also the reliability of
the results.  As indicated in \cite{DGXZ13}, when the scale of problem is large, the numerical error of SDP solver may lead to a result which looks like ``reliable" by the output data while it is indeed unreliable.

\section{Split polynomial} 
\label{sec:mon}

Another new concept, {\em split polynomial}, is introduced in this section. Every convex cover polynomial is a split polynomial. The main results of this section are Theorem \ref{the:splitne} (analogue of Theorem \ref{the:1}) and Theorem \ref{the:split} (analogue of Theorem \ref{cor:2}), which allow one to block-diagonalize a wider class of SDPs.

\begin{definition}
  For a set $\GSUP$ of vectors and any $\va\in \GSUP+\GSUP$, define $\MAP_{\GSUP}(\va)=\{\vb\in \GSUP\mid \exists \vr \in \GSUP, \vb+\vr=\va  \}$.
\end{definition}

\begin{definition}\label{def:vertex}
Suppose ${\GSUP}$ satisfies $\SOSS( p,\GSUP)$ (see Definition \ref{def:soss}) for a polynomial $p$. Define $\VV(p,\GSUP)$ to be the set $\{\va\in {\GSUP}\mid \MAP_{\GSUP}(2\va)=\{\va \} \}$.
\end{definition}

\begin{definition}\label{def:mult}
Suppose ${\GSUP}$ satisfies $\SOSS(p,\GSUP)$ for a polynomial $p$. For any $\va\in {\GSUP}+{\GSUP}$, define
\[\RELAT(\va)=\left\{ \begin{array}{ll}
\{\frac{1}{2}\va\} & \mbox{ if } {\MAP}_{\GSUP}(\va) =\{\frac{1}{2}{\va} \},\\
\bigcup_{\vb,\vr\in \GSUP, \vb\neq\vr, \vb+\vr=\va}  \left( \RELAT(2\vb) \cup \RELAT(2\vr) \right ) & \mbox{ otherwise}.
\end{array}\right.\]
If $\va\notin {\GSUP}+{\GSUP}$, define $\RELAT(\va)=\emptyset$.
\end{definition}

Since $\RELAT(\va)$ is a subset of $\VV(p,\GSUP)$ and obviously $\VV(p,\GSUP)$ is a finite set, Definition \ref{def:mult} makes sense.

\begin{lemma}
  \label{lem:face}
  Suppose ${\GSUP}$ satisfies $\SOSS(p,\GSUP)$ for a polynomial $p$ and $F$ is a face of $\CONV(\GSUP+\GSUP)$. Let $T=\{\va\mid \va\in \VV(p,\GSUP),2\va\in F\},\GSUP_1=(\GSUP+\GSUP)\cap F$. Then $\RELAT(\va)\subseteq T$ for any $\va \in \GSUP_1$.
\end{lemma}

\begin{proof}
  For any  $\vb \in \RELAT(\va)$, by the definition of $\RELAT(\va)$, there are $\vb_1,\cdots,\vb_k, \vr_1,\cdots, \vr_k\in \GSUP$ such that $\vb_i\neq \vr_i$ for $i=1,\cdots,k-1,\va=\vb_1+\vr_1,2\vb_1=\vb_2+\vr_2,\cdots,2\vb_{k-1}=\vb_k+\vr_k, \vb_k=\vr_k=\vb$
and $\RELAT(2\vb)=\{\vb\}$.

 We prove $2\vb_i \in F$ by induction. Because $\va=\vb_1+\vr_1$ and $\va\in F$, we have
$2\vb_1\in F$. Assume that  $2\vb_i\in F$ for  $i<m$. If $i=m$, since
$2\vb_{m-1}=\vb_m+\vr_m$ and $2\vb_{m-1}\in F$, we have $2\vb_m\in F$. Then
$2\vb=2\vb_k\in F$ and hence, $\vb \in T$.  \qed
\end{proof}

\begin{definition}\label{def:ex}
Suppose ${\GSUP}$ satisfies $ \SOSS(p,\GSUP)$ for a polynomial $p$ and $T\subseteq \VV(p,\GSUP)$. Define $\sigma(T)=\{\vr \mid \vr\in {\GSUP}, \RELAT(2\vr)\subseteq T\} $.
\end{definition}

\begin{lemma}  \label{lem:mult}
Suppose $p$ is SOS, say $p=\sum_{i=1}^sh_i^2$, and $\SUPPORT(h_i)\subseteq \GSUP$. Then for any $T\subseteq \VV(p,\GSUP)$ and any $\vb\in {\GSUP}+{\GSUP}$, if $\RELAT(\vb)\subseteq T$, then $\vb \not \in \SUPPORT((p-\sum_{i=1}^s\PROJ(h_i,\sigma(T))^2))$.
\end{lemma}
\begin{proof}
  For any $\vr_1, \vr_2\in {\GSUP} $ with  $ \vr_1+\vr_2=\vb$,  we have $ \RELAT(2\vr_1)\subseteq \RELAT(\vb)$ and  $\RELAT(2\vr_2)\subseteq \RELAT(\vb) $   by the definition of $\RELAT$. Since $\RELAT(\vb)\subseteq T$, we have  $\vr_1,\vr_2\in \sigma(T)$ by the definition of $\sigma(T)$. It is not difficult to see that the coefficient
  of the term $\xx^{\vb}$ in $\sum_{i=1}^s\PROJ(h_i,\sigma(T))^2$ equals that of the term $\xx^{\vb}$ in $\sum_{i=1}^sh_i^2$. Thus, $\xx^{\vb}$ does not appear in $p-\sum_{i=1}^s\PROJ(h_i,\sigma(T))^2$ since $p-\sum_{i=1}^sh_i^2=0$.\qed
\end{proof}


\begin{theorem}\label{the:splitne}
Assume $p=\sum c_{\va}\xx^{\va}$ is SOS, ${\GSUP}$ satisfies $\SOSS(p,\GSUP)$ and $T\subseteq\VV(p,\GSUP)$. If $\RELAT(\va+\vb)\subseteq T$ for any $\va, \vb\in \sigma(T)$, 
then $p_1=\sum_{\va\in \SUPPORT(p), \RELAT(\va)\subseteq T}c_{\va}\xx^{\va}$ is SOS.
\end{theorem}

\begin{proof}
  Suppose $p=\sum_{i=1}^sh_i^2$ and 
  $p_1'=p-p_1$. Set $h_i'=\PROJ(h_i,\sigma(T))$ and $ h_i''=h_i-h_i'$, then
  $p=\sum_{i=1}^s(h_i')^2+2\sum_{i=1}^sh_i'h_i''+\sum_{i=1}^s(h_i'')^2$. By  Lemma \ref{lem:mult},
  $\vb  \not \in  \SUPPORT (p-\sum_{i=1}^s(h_i')^2)$
  for any $\vb \in \SUPPORT(p_1)$, {\it i.e.}, $\SUPPORT(p_1)\cap \SUPPORT (p-\sum_{i=1}^s(h_i')^2)=\emptyset$. Since $\RELAT(\va+\vb)\subseteq T$ for any $\va, \vb\in \sigma(T)$, by the definition of $\sigma(T)$, $\RELAT(\vb)\subseteq T$ for any $\vb \in \SUPPORT(\sum_{i=1}^s(h_i')^2)$.  Thus, $\SUPPORT(p_1')\cap \SUPPORT(\sum_{i=1}^s(h_i')^2)=\emptyset$.
  Summarizing the above, we have
  \begin{enumerate}
  \item $p_1+p'_1=\sum_{i=1}^s(h_i')^2+(p-\sum_{i=1}^s(h_i')^2)$,
  \item $\SUPPORT(p_1)\cap \SUPPORT(p-\sum_{i=1}^s(h_i')^2)=\emptyset$, and
  \item $\SUPPORT(p_1')\cap \SUPPORT(\sum_{i=1}^s(h_i')^2)=\emptyset$.
  \end{enumerate}
  Therefore, $p_1=\sum_{i=1}^s(h_i')^2$.
  \qed
\end{proof}


\begin{definition}[Split polynomial] \label{def:split}
Let ${\GSUP}$ be a finite set satisfying $\SOSS(p,\GSUP)$ for a polynomial $p$. If there exist some pairwise disjoint nonempty
subsets $T_i (i=1,\dots,u)$ of $\VV(p,\GSUP)$ such that
\begin{enumerate}
  \item   $\RELAT(\va+\vb)\subseteq T_i$ for any $\va, \vb\in \sigma(T_i)$(see Definition \ref{def:ex}) for any $i=1,\cdots,u$, and
  \item   for any $\va\in \SUPPORT(p)$, there exist exact one $T_i$ such that $\RELAT(\va)\subseteq T_i$,
\end{enumerate}
 then $p$ is said to be a {\em split polynomial} with respect to $T_1,\cdots,T_u$.
\end{definition}

\begin{theorem}\label{the:split}
Suppose $p=\sum c_{\va}\xx^{\va}$ is a split polynomial with respect to $T_1,\cdots,T_u$, then $p$ is SOS if and only if each $p_i= \sum_{\va\in \SUPPORT(p), \RELAT(\va)\subseteq T_i}c_{\va}\xx^{\va} $ is SOS for $i=1,\cdots,u$.
\end{theorem}

\begin{proof}
Necessity is a direct corollary of Theorem \ref{the:splitne}. For sufficiency, note that the second condition of Definition \ref{def:split} guarantees that $\SUPPORT(p_i)\cap \SUPPORT(p_j)=\emptyset$ for any $i\neq j$ and
$p=\sum_{i=1}^up_i$. 
 \qed
\end{proof}

Now, we give the relation between convex cover polynomial and split polynomial, which indicates that split polynomial is a wider class of polynomials.

\begin{theorem}
  \label{the:consp}
  If $p$ is a  convex cover polynomial, then $p$ is a  split polynomial. The converse is not true.
\end{theorem}

\begin{proof}

  If $p$ is a convex cover polynomial then there exist pairwise disjoint
  faces $F_i (i=1,\dots, u)$  of $\NEW(p)$ such that
  $\SUPPORT(p)\subseteq \cup_{i=1}^uF_i $.  Suppose  $\CONV({\GSUP}+\GSUP)=\NEW(p)$  and  $\GSUP$  satisfies $\SOSS(p,\GSUP)$.
  Let $T_i=\{\va \in \VV(p,\GSUP) \mid 2\va \in F_i \},i=1,\cdots,u$.  We prove that $p$ is a split polynomial with respect to $T_1,\cdots,T_u$.

  We claim that $\sigma(T_j)=\{\vr\in \GSUP \mid 2\vr \in F_j \}$ for $j=1,\cdots, u$. If there exist $\vr_0 \in \sigma(T_j), 2\vr_0 \not \in F_j$, as $F_j$ is a face of $\NEW(p)$, there exist a linear function  $\pi$ such that $\pi(2\vr_0) > \pi(\va)$ for any $\va \in F_j$. By the Definition of $\RELAT(2\vr_0)$, there
  exists $\vb_0 \in \RELAT(2\vr_0)\subseteq T_j$ such that $\pi(2\vb_0)\ge \pi(2\vr_0)$.
  This contradicts with $2\vb_0 \in F_j$. Thus, $\sigma(T_j)\subseteq \{\vr\in \GSUP \mid 2\vr \in F_j \}$.

  We then prove that $ \{\vr\in \GSUP \mid 2\vr \in F_j \}\subseteq  \sigma(T_j)$. Assume that there exists $\vr_0\in \GSUP$ with $ 2\vr_0 \in F_j$
  such that $\vr_0 \not \in \sigma(T_j)$. Then there exists $\vb_0 \in \RELAT(2\vr_0)$ such that $2\vb_0 \not \in F_j$.
  Because $F_j$ is a face of $\NEW(p)$, it is not difficult to see that if
  $\va_1+\va_2\in F_j$ where $\va_1,\va_2\in \GSUP$, then   $2\va_1\in F_j, 2\va_2\in F_j$.
  Therefore,  $2\vb \in F_j$ for any $\vb \in \RELAT(2\vr_0)$,  which  contradicts with $2\vb_0\not \in F_j$.

  Now we have $\sigma(T_j)=\{\vr\in \GSUP \mid 2\vr \in F_j \}$.
  By Lemma \ref{lem:face}, $\RELAT(\va+\vb)\subseteq T_j$ for any $\va, \vb\in \sigma(T_j)$.  Since  $\SUPPORT(p)\subseteq \cup_{i=1}^uF_i $ and $F_i$ are pairwise disjoint,
  there exists  exact  one $T_i$  such that  $\RELAT(\va)\subseteq T_i$ for any $\va\in \SUPPORT(p)$. As a result, $p$ is a split polynomial with respect to $T_1,\cdots,T_u$.

  Note that the Motzkin polynomial in  Example \ref{ex:motzkin}  is a  split polynomial but not a convex cover polynomial since $x_1^2x_2^2$ lies in the interior of
  $\NEW(p)$ (see Figure \ref{fig:motzkin}).  That means  the converse is not true. \qed
\end{proof}

\begin{remark}
  One may wonder under what condition a split polynomial is a convex cover polynomial.
  A reasonable conjecture may be as this:

  {\em Let ${\GSUP}$ be a finite set satisfying $\SOSS(p,\GSUP)$ with $\CONV(\GSUP+\GSUP)=\NEW(p)$ for a polynomial $p$. If $\VV(p,\GSUP)$
  contains only vertices of   $\CONV(\GSUP)$, then
   $p$ is a split polynomial if and  only if   $p$ is a convex convex polynomial.}

 Unfortunately, the conjecture is not true. For example,
 let $$p=x_1^4x_2^2x_3^2+x_1^2x_2^4x_3^2-2x_1^2x_2^2x_3^2 +x_3^2+x_1^2x_2^2+x_1^2x_2^2x_3^4,$$
 then $$\begin{array}{rl}
   \GSUP&=\{(2,1,1),(1,2,1),(1,1,1),(0,0,1),(1,1,0),(1,1,2)\},\\
   \VV(p,\GSUP)&=\{(2,1,1),(1,2,1),(0,0,1),(1,1,0),(1,1,2)\}.
 \end{array}
 $$
 Obviously, $\VV(p,\GSUP)$ contains only vertices of $\CONV(\GSUP)$. Set $T_1=\{(2,1,1),(1,2,1),(0,0,1)\},$ $ T_2=\{(1,1,0),(1,1,2)\} $, then it is easy to check $p$ is a split polynomial with respect to $T_1,T_2$. But $p$ is not a convex cover polynomial by Proposition \ref{lem:int} because $x_1^2x_2^2x_3^2$ lies in the interior of $\NEW(p)$.
 
 The example indicates that the relation between split polynomial and convex cover polynomial may be complicated. We do not find a good sufficient condition for a split polynomial to be a convex cover polynomial.
\end{remark}

\section{Algorithm} 
\label{sec:imp}

Existing SDP based SOS solvers consists of the following two main steps: computing a set $\GSUP$ which satisfies $\SOSS(p,\GSUP)$ for a given $p$; solving the feasibility problem of (\ref{eq:SOS1}) related to $\GSUP$ by SDP solvers.
In this section, we give a new algorithm (Algorithm \ref{alg:main}) for SOS decomposition. The algorithm employs the following strategies. First, we give a different technique for computing an initial set $\GSUP$ which satisfies $\SOSS(p,\GSUP)$ for a given $p$. Second, we check one necessary condition (Lemma \ref{lemcor:4}) to refute quickly some non-SOS polynomials. Third,
if the input polynomial is detected to be a split polynomial, we reduce the problem into several smaller sub-problems based on Theorem \ref{the:split}. This section is dedicated to describe the strategies in detail and the performance of the algorithm is reported in the next section.

We first describe the new technique for computing an initial set $\GSUP$.
The following lemma is a direct corollary of the result in \cite{renz78} (see also Eq. (\ref{eq:sup}) in Section \ref{sec:pro}).
\begin{lemma}
  Suppose $p$ is a polynomial and $\vr$ is a given vector. Let $c=\max_{\va\in \frac{1}{2}\PSUPE} \vr^T\va$. For any $\GSUP$ which satisfies $\SOSS(p,\GSUP)$, after deleting every $\vb$ in $\GSUP$ such that $\vr^T\vb>c$, $\SOSS(p,\GSUP)$ still holds.
  \label{lem:2}
\end{lemma}

By Lemma \ref{lem:2}, it is easy to give a method for computing an initial set $\GSUP$ which satisfies $\SOSS(p,\GSUP)$ for a given $p$. That is, first choose a coarse set $\GSUP$
which satisfies $\SOSS(p,\GSUP)$, {\it e.g.}, the set defined be Eq. (\ref{eq:subg}); 
then prune the superfluous elements in $\GSUP$ by choosing {\em randomly} $\vr$. This is indeed a common method in existing work \cite{PPP02,lofberg2004yalmip,seiler2013sosopt}.

We employ a different strategy to construct an initial $\GSUP$ satisfying $\SOSS(p,\GSUP)$.
The procedure is as follows.
For a given polynomial $p$, firstly, we compute the set $\frac{1}{2}\PSUPE$ (recall that ${\PSUPE}={\PSUP}\cap(2\ZZ^n)$ where $\PSUP$ is the support of $p$) and an over approximation set $\GSUP$ of integer points in
$\CONV(\frac{1}{2}\PSUPE)$. Secondly, let $B$ be the matrix whose columns are all the vectors  of $\frac{1}{2}\PSUPE$.
We choose one by one the hyperplanes whose normal directions are the eigenvectors of $BB^T$ to delete superfluous lattice points in $\GSUP$ by Lemma \ref{lem:2}.

\begin{definition}
\label{def:apca}
We denote by $\algpca(p)$ the above procedure to compute an initial $\GSUP$ satisfying $\SOSS(p,\GSUP)$ for a given polynomial $p$.
\end{definition}

We cannot prove the above strategy is better in general than the random one. However, inspired by principal component analysis (PCA), we believe in many cases the shape of  $\CONV(\frac{1}{2}\PSUPE)$ depends on eigenvectors of $BB^T$. On a group of randomly generated examples (see Example \ref{ex:sqr}), we show that the size of $\GSUP$ obtained by using random hyperplanes to delete superfluous lattice points are $10\%$ greater than that of the output of our algorithm $\algpca$ (see Figure \ref{fig:pcacmp}).

\begin{example}
\label{ex:sqr}
$SQR(k,n,d,t)=g_1^2+\dots+g_k^2$ where  $\deg(g_i)=d$, $\#(\SUPPORT(g_i))=t$, $\#({\rm var}(g_i))=n$. 
\end{example}

\begin{figure}[ht]
	\centering
	\includegraphics[width=0.5\textwidth]{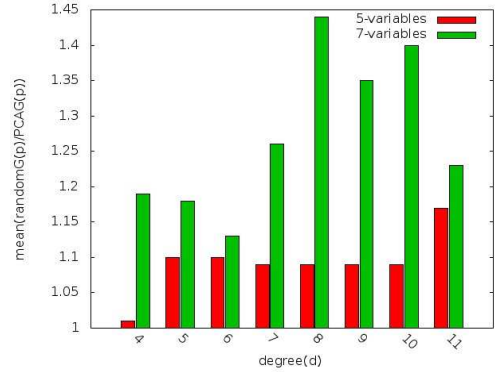}
	\caption{ Mean ratio of $\#(\GSUP)$  between  random algorithm and \algpca(p)   on every random group $SQR(k,n,d,t).$ \label{fig:pcacmp} Red color corresponds to $k=4,n=5,t=3$ and green color corresponds to $k=5, n=7,t=4$. For any given $(k,n,d,t)$,  we generate $10$ polynomials randomly.
}
\end{figure}

\begin{lemma}\cite{kim05,seiler2013simplification}
  For a polynomial $p$ and a set $\GSUP$ which satisfy $\SOSS(p,\GSUP)$, after deleting every element $\va$ in $\GSUP$ which satisfies that $2\va\not \in \PSUPE$ and $\MAP_{\GSUP}(2\va)=\{\va\}$, the relation $\SOSS(p,\GSUP)$ still holds.
  \label{lem:reduce}
\end{lemma}

\begin{definition}
\label{def:aea}
 We denote by $\algexa(p)$ the procedure which deletes superfluous elements of the output of $\algpca(p)$ based on Lemma \ref{lem:reduce}.
\end{definition}

The following lemma is a simple but very useful necessary condition which can detect non-SOS polynomials efficiently in many cases.
\begin{lemma}
  \label{lemcor:4}

  Suppose ${\GSUP}$ satisfies $ \SOSS(p,\GSUP)$ for a polynomial $p$. If $p$ is SOS, then $\va \in \GSUP+\GSUP$ for any $\va \in \SUPPORT(p)$.

\end{lemma}
\begin{proof}
  If $p$ is SOS, since $p$ and $ \GSUP$ satisfy  relation  $SOSS(p,\GSUP)$, there are $q_1,\ldots,q_s$ such that $p=\sum_{i=1}^sq_i^2$ and $\SUPPORT(q_i)\subseteq \GSUP$. Hence, for every monomial $\xx^{\va}$ of $p$ there are $q_i$, $\xx^{\vb}, \xx^{\vr}$ such that $\xx^{\vb},\xx^{\vr}$ are monomials of $q_i$ and $\xx^{\va}=\xx^{\vb}\xx^{\vr}$. Therefore, $\va\in \GSUP+\GSUP$ for any $\va\in \SUPPORT(p)$. \qed

\end{proof}
\begin{example}\cite{choilam77}
  Let $q(x,y,z)=1+x^2y^2+y^2z^2+z^2x^2-4xyz.$ It is easy to know that
  $$\begin{array}{l}
  \frac{1}{2}\PSUPE=\{(0,0,0),(1,1,0),(1,0,1),(0,1,1) \},\\
  \GSUPZ=\{(0,0,0), (1,0,0), (0,1,0), (0,0,1), (1,1,0), (1,0,1), (0,1,1) \},
  \end{array}
  $$
  and $\SOSS(q,\GSUPZ)$ holds. By Lemma \ref{lem:reduce}, after deleting $(1,0,0), (0,1,0), (0,0,1)$  from $\GSUPZ$, we have   $\GSUP=\algexa$$(q)=\{(0,0,0),(1,1,0),(1,0,1),(0,1,1)\}$ and $\SOSS(q,\GSUP)$ holds. 
  Since $(1,1,1)\notin \GSUP+\GSUP$, by Lemma \ref{lemcor:4}, $p$ is not SOS.
\end{example}

For an input polynomial $p$, by setting $\GSUP=\algexa(p)$, we obtain a set $\GSUP$ satisfying $\SOSS(p,\GSUP)$. Now, we check whether or not $p$ is a split polynomial related to this $\GSUP$. And if it is, the original problem can be reduced to several smaller sub-problems. The details are described formally as Algorithm \ref{alg:relate} and Algorithm \ref{alg:main}.

\begin{algorithm}
  \SetAlgoCaptionSeparator{.}
  \caption{\algrelate \label{alg:relate} }
  \DontPrintSemicolon
  \KwIn{ $p\in \QQ[\xx]$  }
  \KwOut{ The map $\RELAT$ defined by Definition \ref{def:mult}}

  ${\GSUP}=\algexa(p)$;\;
  Let $C$ be a map from $\GSUP$ to $\{$ true, false $ \}$;\;
  \lFor {$\va\in  {\GSUP }$ }{
    $C(\va)$=false;\;
  }

  Let $\VV(p,\GSUP)$ be the set defined by Definition \ref{def:vertex};\;

  Initialize  $\RELAT(\va)=\emptyset$ for any $\va\in \ZZ^n$;\; 
  \For {$\va\in \VV(p,\GSUP)$ }{
    $\RELAT(2\va)=\{\va\}$;\;
    $C(\va)$=true;\;
  }
  Let run=true;\;
  \While{run}{
    run=false;\;
    \For{$\va\in \GSUP$}{
      \If{ $C(\va)$}{
        $C(\va)$=false;\;
        \For {$\vb\in \GSUP$}{
          \If{ $\RELAT(2\va )\not \subseteq \RELAT( \va+\vb)$}
          {
            $\RELAT(\va+\vb)=\RELAT(\va+\vb)\cup  \RELAT(2\va)  $;\;
            \If{$ \va+\vb\in 2\ZZ^n$ }{
              run=true;
              $C( (\va+\vb)/2)$=true;\;
            }

          }
        }
      }
    }
  }

  \Return { $\RELAT$;}
\end{algorithm}

\begin{algorithm}[H]
  \SetAlgoCaptionSeparator{.}
  \caption{\algf \label{alg:main} }
  \DontPrintSemicolon
  \KwIn{ $p\in \QQ[\xx]$}
  \KwOut{  {\bf false that means $p$ is not SOS; or $\{q_1,\dots,q_s\}$ where $p,q_i$ satisfy Eq. (\ref{eq:sum}) numerically}}

  Let $\RELAT $ be the output of $\algrelate(p)$;\;

  \lFor {$\va\in \SUPPORT(p)$}{
    \lIf{ $\va\not \in \GSUP+\GSUP$}{
      \Return {false}; \ \ \ \ // Lemma \ref{lemcor:4}  \;
    }
  }

  \For {$\va\in \SUPPORT(p)$}{
    \If{ $p$ is a split polynomial with respect to $\RELAT(\va)$ }{
      Let $p_1,p_2$ be as in Theorem \ref{the:split};\;
      Let $R_1$ be the output of $\algf(p_1)$;\;
      Let $R_2$ be the output of $\algf(p_2)$;\;
      \lIf{ $R_1$ or $R_2$ is false}{ \Return{false};\;}
      \Return{ $R_1\cup R_2$;}
    }
  }
  \Return {$\algSOS(p,{\GSUP})$;\ \ \ \ // Notation \ref{not:1} }
\end{algorithm}

\begin{example}
We illustrate \algf\ on the polynomial $p$ in Example \ref{ex:2}. First,
\[\begin{array}{rl}
\SUPPORT(p) & =\{(0,6),(6,0),(0,4),(4,0),(2,2) \},\\
\GSUP & =\{(0,2),(0,3),(1,1),(1,2),(2,0),(2,1),(3,0) \},\\
\VV(p,\GSUP) & =\{(0,2),(0,3),(2,0),(3,0) \}.
\end{array}\]

Second,
$\RELAT((0,4))=\{(0,2) \}$, $\RELAT((0,6))=\{(0,3) \} $,  $\RELAT((4,0))=\{(2,0) \} $,  $\RELAT((6,0))=\{(3,0) \} $,  $\RELAT((2,2))=\{(0,2),(2,0) \} $. Set $T=\RELAT((2,2))=\{(2,0),(0,2)\},$ it is easy to see that $p$ is a split polynomial with respect to $T$ and $p_1=x_1^4-2x_1^2x_2^2+x_2^4, p_2=x_1^6+x_2^6$.

Third, similarly, $\algf(p_2)$ divides $p_2$ into $p_{21}=x_1^6, p_{22}=x_2^6$. Finally, $\algf(p)$ outputs ``$1.0000*x_2^4 +1.0000*x_2^6 -2.0000*x_1^2*x_2^2 +1.0000*x_1^4 +1.0000*x_1^6=
(-1.00*x_2^2+1.00*x_1^2)^2+(1.00*x_2^3)^2+(1.00*x_1^3)^2$".
\end{example}

\begin{example}[Motzkin polynomial]\label{ex:motzkin}
  $f=x_1^4x_2^2+x_1^2x_2^4-3x_1^2x_2^2+1$.

Because $\SUPPORT(f)=\{(4,2),(2,2),(2,4),(0,0)\}$ and ${\mathbf Q}=\{(0,0),(1,1),(2,1),(1,2)\}$,\\  $\algrelate(f)$ returns $\RELAT((4,2))=\{(2,1)\}$, $\RELAT((2,4))=\{(2,4)\}$, $\RELAT((2,2))=\{(1,1)\}$, $\RELAT((0,0))=\{(0,0) \}$. Then \algf($f$)\ will return false when it reaches line $8$ for $\va=(2,2)$.
\end{example}

\begin{remark}
Let $\GSUP=\algexa(p)$. By Definition \ref{def:split}, to determine whether $p$ is a split polynomial, one should check all the non-empty subsets of $\VV(p,\GSUP)$. However, this approach is obviously inefficient.
Therefore, in Algorithm \ref{alg:main} we only check whether $p$ is a split polynomial with respect to  $\RELAT(\va)$ for some $\va\in \SUPPORT(p)$. Although this incomplete check may miss some split polynomials, as is shown in the next section, it is effective in many cases.
\end{remark}

\begin{figure}[ht]
  \begin{minipage}[b]{0.30\linewidth}
    \centering
    \psset{radius=1cm}
    \psset{xunit=0.7,yunit=0.7}
    \begin{pspicture}(-0.3,-0.3)(4.5,4.5)
      \psaxes{->}(0,0)(-0.3,-0.3)(4.5,4.5)
      \psaxes[
      dy=1,
      Dy=1,
      labels=y,
      ticks=y,
      ](0,0)(-0.3,-0.3)(4.5,4.5)

      {\psset{fillstyle=ccslope,slopebegin=cyan!40,slopeend=yellow}
        \cnodeput(0,0){A}{\strut\boldmath$1$}
        \cnodeput(4,2){B}{{\tiny \strut\boldmath$x_1^4x_2^2$}}
        \cnodeput(2,4){C}{{\tiny \strut\boldmath$x_1^2x_2^4$}}
        \cnodeput(2,2){D}{{\tiny \strut\boldmath$x_1^2x_2^2$}}

      }

      \arc{A}{B}{120}
      \arc{B}{C}{650}
      \arc{C}{A}{780}
    \end{pspicture}
    \caption{Newton polytope of  Example \ref{ex:motzkin}. \label{fig:motzkin}}
  \end{minipage}
  \hspace{0.3cm}
  \begin{minipage}[b]{0.30\linewidth}
    \centering
    \psset{xunit=1.4,yunit=1.4}
    \begin{pspicture}(-0.36,-0.36)(2.3,2.3)
      \psaxes{->}(0,0)(-0.36,-0.36)(2.3,2.3)
      \psaxes[
      dy=1,
      Dy=1,
      labels=y,
      ticks=y,
      ](0,0)(-0.3,-0.3)(2.3,2.3)
      {\psset{fillstyle=ccslope,slopebegin=cyan!40,slopeend=yellow}
        \cnodeput(0,0){A}{\strut\boldmath$1$}
        \cnodeput(2,1){B}{{\tiny \strut\boldmath$x_1^2x_2$}}
        \cnodeput(1,2){C}{{\tiny \strut\boldmath$x_1x_2^2$}}
        \cnodeput(1,1){D}{{\tiny \strut\boldmath$x_1x_2$}}

      }

      \arc{A}{B}{120}
      \arc{B}{C}{650}
      \arc{C}{A}{780}

    \end{pspicture}

    \caption{${\GSUP}^0$ of  Example \ref{ex:motzkin}.}
    \label{fig:2}
  \end{minipage}	
  \hspace{0.3cm}
  \begin{minipage}[b]{0.30\linewidth}
    \centering
    \psset{xunit=1.4,yunit=1.4}
    \begin{pspicture}(-0.14,-0.14)(2.3,2.3)

      \psaxes{->}(0,0)(-0.14,-0.14)(2.3,2.3)
      \psaxes[
      dy=1,
      Dy=1,
      labels=y,
      ticks=y,
      ](0,0)(-0.3,-0.3)(2.3,2.3)

      {\psset{fillstyle=ccslope,slopebegin=cyan!40,slopeend=yellow}
        \cnodeput(0,0){A}{\strut\boldmath$1$}
        \cnodeput(2,1){B}{{\tiny \strut\boldmath$x_1^2x_2$}}
        \cnodeput(1,2){C}{{\tiny \strut\boldmath$x_1x_2^2$}}
        \cnodeput(1,1){D}{{\tiny \strut\boldmath$x_1x_2$}}
      }

      \arc{A}{B}{120}
      \arc{B}{C}{650}
      \arc{C}{A}{780}

    \end{pspicture}

    \caption{The output of $\algexa$ for  Example \ref{ex:motzkin}.}
    \label{fig:3}
  \end{minipage}	

\end{figure}

\section{Experiments}
\label{sec:exp}
The above algorithms have been implemented as a {\tt C++} program, \aisat 
.
Compilation has been  done using g++ version 4.6.3 with optimization flags -O2.
We use {\tt Singular} \cite{singular} to read polynomials from files or standard input and use {\tt Csdp} \cite{csdp}
as SDP solver. The program has been tested on many benchmarks in the literature and on lots of examples
generated randomly.

We report in this section corresponding experimental data of our program and some well-known SOS solvers, such as
\yalmip, \sostool, \sosopt.
  The matlab version is
R2011b and \sostool's version is 3.00. Both \yalmip\ and \sosopt\ are   the latest release.  The SDP solver of \yalmip, \sostool\ and \sosopt\  is  SeDuMi 1.3.

All the numerical examples listed were computed on a 64-bit Intel(R) Core(TM) i5 CPU 650 @ 3.20GHz with 4GB RAM memory and Ubuntu 12.04 GNU/Linux.

\subsection{Examples}
In this subsection, we define four classes of examples. The first class of examples are modified
from \cite{Parrilo00}, which are positive but not necessarily SOS. The second one is from \cite{Haglund97theoremsand,KYZ09}. The other two classes are sparse polynomials randomly generated by Maple's command {\tt randpoly} where
the third class of polynomials are constructed in the form of SOS.

The number of  elements in a set $Q$ is denoted by  $\#(Q)$, $\deg(p)$ denotes the total degree of a polynomial $p$, ${\rm var}(p)$ denotes the set of variables occuring in a polynomial $p$.

\subsubsection{$B_m$}
$B_m=(\sum_{i=1}^{3m+2}x_i^2)^2-2\sum_{i=1}^{3m+2}x_i^2\sum_{j=1}^mx_{i+3j+1}^2,$
where $x_{3m+2+r}=x_r$. $B_m$ is modified from \cite{Parrilo00}. For any $m\in \ZZ$, $B_m$ is homogeneous and is  a positive polynomial.


\subsubsection{$p_{i,j}$}
Monotone Column Permanent (MCP) Conjecture was given in \cite{Haglund97theoremsand}. When $n=4$, this Conjecture is equivalent to decide whether $p12, p13, p22, p23$ are positive polynomials and this case has been studied in \cite{KYZ09}. \footnote{The polynomials can be found at \url{http://www4.ncsu.edu/~kaltofen/software/mcp_conj_4/}.}


\subsubsection{$SQR(k,n,d,t)$ (see Example \ref{ex:sqr})}
$SQR(k,n,d,t)=g_1^2+\dots+g_k^2$ where  $\deg(g_i)=d$, $\#(\SUPPORT(g_i))=t$, $\#({\rm var}(g_i))=n$. 

\subsubsection{$RN(n,d)$}

$RN(n,d)= g_1^2+ g_2\sum_{i=1}^nx_i+ 100g_3^2+100 $, where  $\deg(g_1)=d, \deg(g_2)=d-3, \deg(g_3)=d-2, {\rm var}(g_i)=\{x_1,\ldots,x_n\}$. For any given  $(n,d)$ where $n\in \{5,10\}$ and $4\le  d\le 12$,  we generate $10$ corresponding polynomials.

\subsection{Results}

If we only compare the timings of different tools,
the  comparison is somehow unfair since the implementation languages are different. Since the main idea of this paper is to compute smaller set $\GSUP$ for given polynomial $p$ which make relation $\SOSS(p,\GSUP)$ hold, we also report the comparison  of the size of $\GSUP$ computed by different tools. It is reasonable to believe that  the total time of computing SOS decomposition becomes  shorter as the size of  $\GSUP$ getting   smaller if we use the same SDP solver and the cost of computing smaller $\GSUP$ is not expensive.
In fact, for all the following examples except $B_m$, the time taken in computing $\GSUP$ by \aisat\ is less than $0.1$ seconds.

We explain the notations in the following tables.
 Each $(b,s)$  for  \aisat's $\#(\GSUP)$ means \aisat\ divides  the polynomial into $b$ polynomials$p_1,\ldots,p_b$ and  $s$ is the largest number of  $\#(\GSUP_i)$ corresponding to $p_i$.
The ``\textemdash"   denotes that there
is no corresponding output. 

\begin{table}
\begin{center}
	\caption{The results on $B_m$. \label{tab:1}}
\begin{tabular}{l*{9}{c}c}
	\hline
	& 	\multicolumn{5}{c|}{$\#(\GSUP)$}     & \multicolumn{5}{c}{time(s)}
	\\
	\cline{2-11}
	Tools      & $B_1$ 	 & $B_2$ &  $B_3$ &  $B_4$ &  $B_5$ & $B_1$  & $B_2$ &  $B_3$ &  $B_4$ &  $B_5$   \\
	\hline
	\aisat   	&  (5,1)	 &  (1,33)  &  (1,55) &  (1,94)  &  (1,150)  & 0.00 & 0.15 & 0.62  & 6.79   &73.27  \\
	\yalmip		& 15 & 36 & 66 & 105 & 153   & 0.24 & 0.53 & 0.72 & 1.26   & 167.21    \\
	\sostool    & 15 & 36 & 66 & 105 & 153   & 0.30 & 0.42 & 2.09 & 21.32  & 163.05    \\
	\sosopt     & 15 & 36 &  \textemdash  & \textemdash     &  \textemdash   & 0.25 & 3.01	   & error & error & error    \\
        \hline
\end{tabular}
\end{center}
\end{table}

The results on $B_m$ by these tools are listed in Table \ref{tab:1}. The polynomials $B_1$ and $B_2$ are SOS, the others are not.  All the above tools except \sosopt\ give  correct\footnote{ The meaning of correction is that the output is right with respect to   a certain numerical error.} outputs on $B_m$.
 Although $B_i$ is  not a sparse polynomial, our algorithm can also reduce  $\#(\GSUP)$.
 When the size of polynomial is large, \sosopt\  takes so much time to solve it. This phenomenon also  occurs in the following examples. For convenience, we do not list the results of
 \sosopt\ in the following.

The results on $p_{i,j}$ by those tools are listed in Table \ref{tab:0}.

\begin{table}
	\caption{The results on  $p_{i,j}$. \label{tab:0}}
\begin{tabular}{l*{7}{c}c}
	\hline
	& 	\multicolumn{4}{c|}{$\#(\GSUP)$}     & \multicolumn{4}{c}{time(s)}
	\\
	\cline{2-9}
	Tools       & $p_{1,2}$ 	 & $p_{1,3}$ &  $p_{2,2}$ &  $p_{2,3}$  & $p_{1,2}$  & $p_{1,3}$ &  $p_{2,2}$ &  $p_{2,3}$    \\
	\hline
	\aisat   	&  (1,77)	 &  (5,15)  &  (1,62) &  (6,39)  &  1.98  & 0.01   & 1.25  & 0.19  \\
	\yalmip		& 77             & 29        & 62      & 53      & 4.93   & 1.81   & 4.97 & 4.10     \\
	\sostool        & wrong             & wrong       & 62       & wrong     & wrong     & wrong   & 3.77 & wrong    \\

        \hline
\end{tabular}
\end{table}

Table \ref{tab:2} lists the results on  examples $SQR$ (see Example \ref{ex:sqr}). We randomly generate $10$ polynomials for every $(k,n,d,t)$. All the outputs of \aisat\ and \yalmip\  are correct. Some data  corresponding to $\sostool$ are ``wrong'',
 which means that $\sostool$'s  output is wrong or there is an error occurred during its execution. For many examples of $SQR$,  $\aisat$ can divide the original polynomial into some simpler polynomials.  
 By the complexity analysis in Section \ref{sec:pro}, this division can greatly improve
 efficiency. 
 
 We demonstrate this fact by one  polynomial of group $SQR(4,5,10,3)$.
 \begin{example}
\label{ex:6}
	 $p=(-91w^4x^2yz^3-41k^4xy^2z^2-14kwx^3y^2z)^2+(-40kx^7yz+16w^    4xy+65w^2y^4)^2+(11kx^2y^6z-34k^5x^3z-18kyz^5)^2+(-26k^4w^3    xyz-35xy^6z^3-57kw^2x^2z^3)^2.$
 \end{example}
 \begin{remark}
   As explained before, $SQR$ is constructed in the form of SOS. But the polynomial is expanded
 before input to  the tools.
 \end{remark}
 
In Example \ref{ex:6}, $\aisat$ divides $p$ into four simpler polynomials $p_1,p_2,p_3,p_4$. For  each simpler polynomial $p_i$, $\aisat$ constructs a  set $\GSUP_i$  whose  size is $3$ and $\SOSS(p_i,\GSUP_i)$ holds.
 $\yalmip$ constructs one $\GSUP$ for $p$ whose size is $97$ and $\sostool$ also constructs one $\GSUP$ for $p$ whose size is
 $104$. If the time   consumed by  constructing  $\GSUP$ is short compared with total time and
assume these three tools use the same SDP solver,
 the ratio of total time of three tools is $4(3^{2+2c}): 97^{2+2c}: 104^{2+2c}$ where $ 1\le c\le2$. In fact, in our experiments, the total time of these
 three tools on this example is $0.02$ seconds, $23.91$ seconds and $48.47$ seconds, respectively.

In addition to efficiency, correction is also important. Figure \ref{fig:error} shows the number of ``wrong" of \sostool\ on every group of random polynomials $SQR$. As explained above, ``wrong" means that $\sostool$'s output is wrong or there is an error occurred during its execution. Those ``wrong"s are caused by numerical instability. Therefore, the number of ``wrong'' increases with the increase of the problem's size.

 The above experiments are all about polynomials which are SOS. Figure \ref{fig:rn} is about
 timings for refuting non-SOS  polynomials.
 {\em For all $180$ $RN$ polynomials, \aisat\ takes $1.07$ seconds to refute all of them.} And there are polynomials
 in these $180$ polynomials on  which \sostool\ cannot finish execution within $10000$ seconds. So we do not list its output.
 Figure $\ref{fig:rn}$ is the mean time of \yalmip\ for every group of polynomials.

{\begin{table}\footnotesize
	\caption{ $\#(\GSUP) $ of  random polynomials  $SQR(k,n,d,t)$ \label{tab:2}}
\begin{tabular}{l*{10}c}
	\hline
	Tools     & 1 & 2 &  3 & 4  & 5 &  6 &  7 & 8 & 9 & 10       \\
	\hline
	& 	\multicolumn{10}{c}{$\#(\GSUP) $}
	\\
	\cline{2-10}
	\hline
	& 	\multicolumn{10}{c}{$ k=4,n=5,d=5,t=3 $}
	\\
	\cline{2-10}
	\hline

	\aisat   	& (2,15) & (1,44) & (2,11) & (4,4)  & (4,4)  & (1,25)		& (4,5) & (3,9)	& (2,8)	& (2,20)  \\
	\yalmip		& 24	 & 45	  & 33     & 18     & 23     & 36		    & 22    & 20	& 15	& 25   \\
	\sostool    & 24     & 45     & 33     & 18     & 23     & 36           & 22    & 20    & 15    & 25  \\
	\hline
	& 	\multicolumn{10}{c}{$ k=4,n=5,d=10,t=3 $}
	\\
	\cline{2-10}
	\hline

	\aisat   	& (4,3) & (4,3) & (4,10) &(3,6)   & (2,7)  & (4,4)	& (4,3) & (4,3)	& (4,3)	& (2,26)  \\
	\yalmip		& 97 & 91 & 42 & 23  & 45  & 40		& 101 & 62	& 95	& 52   \\
	\sostool    & 104 & 94 & 36 & 23  & 48  & 41     & 109 & 70   & 104    & 52  \\

	\hline
	& 	\multicolumn{10}{c}{$ k=5,n=7,d=5,t=4 $}
	\\
	\cline{2-10}
	\hline

	\aisat   	& (4,7) & (5,5) & (2,13) & (5,4)  & (4,11)  & (5,4)	& (4,7) & (3,12)	& (5,5)	& (4,10)  \\
	\yalmip		& 21 & 33 & 24 & 24  & 28  & 24	& 21 & 28	& 42	& 33   \\
	\sostool    & wrong & 33 & 24 & 24  & 28  & 24     & 21 & 28   & 42    & 33  \\
	\hline
	& 	\multicolumn{10}{c}{$ k=5,n=7,d=10,t=4 $}
	\\
	\cline{2-10}
	\hline

	\aisat   	& (5,4) & (5,4) & (5,5) & (5,5)  & (5,4)  & (5,4)	& (5,4) & (5,4)	& (3,11)	& (5,4)  \\
	\yalmip		& 45 & 82 & 74 & 59  & 48  & 70	& 79 & 63	& 41	& 57   \\
	\sostool    & wrong & wrong& wrong  & 63  & 57     & 76 & wrong  & 67 & wrong   &wrong  \\

	& 	\multicolumn{10}{c}{$ k=5,n=7,d=5,t=6 $}
	\\
	\cline{2-10}
	\hline
	\aisat   	& (1,26) & (1,29) & (1,28) & (1,72)  & (1,37)  & (1,30)	& (1,27) & (4,7)	& (2,14)	& (1,61)  \\
	\yalmip		& 28 & 38 & 28 & 82  & 48  & 31	& 33 & 34	& 34	& 69   \\
	\sostool    & wrong & 38& 28  & 82  & wrong     & 31 & 33  &wrong & 34   &wrong  \\

	& 	\multicolumn{10}{c}{$ k=5,n=7,d=8,t=6 $}
	\\
	\cline{2-10}
	\hline
	\aisat   	& (4,7) & (4,6) & (4,7) & (4,7)  & (4,6)  & (4,6)	& (4,6) & (2,24)	& (4,6)	& (4,6)  \\
	\yalmip		& 38 & 34 & 71 & 121  & 51  & 57	& 75 & 100	& 47	& 29   \\
	\sostool    & 39 & wrong& wrong  & 128  & 67     & 67  &78 & 111   &52 & 31  \\
	\hline

\end{tabular}
\end{table}
}

\begin{figure}[ht]
	\centering
	\includegraphics[width=0.6\textwidth]{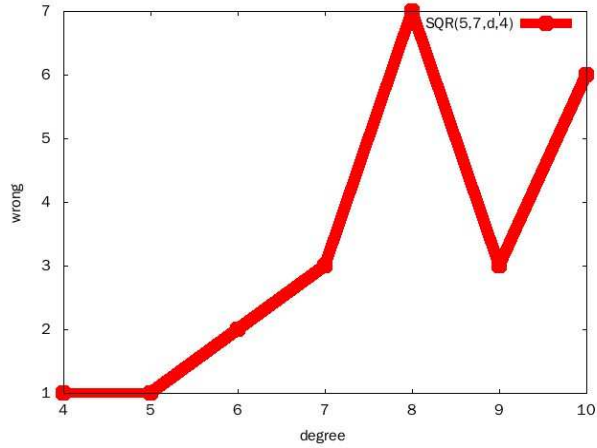}
	\caption{ The number of ``wrong" of \sostool\ on every group of random polynomials. \label{fig:error}}
\end{figure}

\begin{figure}[ht]
	\centering
	\includegraphics[width=0.6\textwidth]{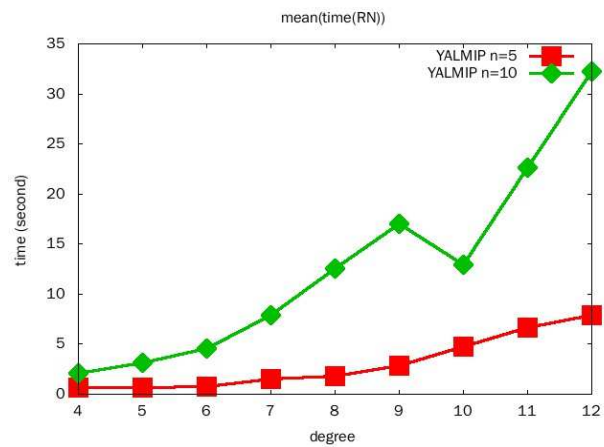}
	\caption{Mean running time of \yalmip\ on every group of $RN$ polynomials. \label{fig:rn}}
\end{figure}

\section*{Acknowledgements}
The authors would like to thank Prof. Hoon Hong for his suggestion and discussion on an early version of this paper and thank Xiaojuan Wu for helping improve the presentation of the paper.

The authors are very much grateful to the reviewer whose constructive suggestions and insightful comments not only help improve greatly  the presentation of the paper but also lead to some new results.


\begin{thebibliography}{10}
\providecommand{\url}[1]{{#1}}
\providecommand{\urlprefix}{URL }
\expandafter\ifx\csname urlstyle\endcsname\relax
  \providecommand{\doi}[1]{DOI~\discretionary{}{}{}#1}\else
  \providecommand{\doi}{DOI~\discretionary{}{}{}\begingroup
  \urlstyle{rm}\Url}\fi

\bibitem{artin27}
Artin, E.: ﾃ彙er die zerlegung definiter funktionen in quadrate.
\newblock Abhandlungen aus dem Mathematischen Seminar der Universitﾃ､t Hamburg
  \textbf{5}(1), 100--115 (1927).
\newblock \doi{10.1007/BF02952513}.
\newblock \urlprefix\url{http://dx.doi.org/10.1007/BF02952513}

\bibitem{blekherman2006there}
Blekherman, G.: There are significantly more nonegative polynomials than sums
  of squares.
\newblock Israel Journal of Mathematics \textbf{153}(1), 355--380 (2006)

\bibitem{csdp}
Borchers, B.: Csdp, a c library for semidefinite programming.
\newblock Optimization Methods and Software \textbf{11}(1-4), 613--623 (1999)

\bibitem{choilam77}
Choi, M.D., Lam, T.Y.: Extremal positive semidefinite forms.
\newblock Mathematische Annalen \textbf{231}(1), 1--18 (1977)

\bibitem{choi1995sums}
Choi, M.D., Lam, T.Y., Reznick, B.: Sums of squares of real polynomials.
\newblock In: Proceedings of symposia in pure mathematics, vol.~58, pp.
  103--126. American Mathematical Society (1995)

\bibitem{DGXZ13}
Dai, L., Gan, T., Xia, B., Zhan, N.: Barrier certificates revisited.
\newblock arXiv preprint arXiv:1310.6481  (2013)

\bibitem{singular}
Greuel, G.M., Pfister, G.: Singular and applications.
\newblock Jahresbericht der \textbf{1505}(108), 4 (2006)

\bibitem{Haglund97theoremsand}
Haglund, J., Ono, K., Wagner, D.G.: Theorems and conjectures involving rook
  polynomials with real roots.
\newblock In: in Proc. Topics in Number Theory and Combinatorics, State, pp.
  207--221 (1997)

\bibitem{kaltofen2008exact}
Kaltofen, E., Li, B., Yang, Z., Zhi, L.: Exact certification of global
  optimality of approximate factorizations via rationalizing sums-of-squares
  with floating point scalars.
\newblock In: Proceedings of the twenty-first international symposium on
  Symbolic and algebraic computation, pp. 155--164. ACM (2008)

\bibitem{KYZ09}
Kaltofen, E., Yang, Z., Zhi, L.: A proof of the monotone column permanent (mcp)
  conjecture for dimension 4 via sums-of-squares of rational functions.
\newblock In: Proceedings of the 2009 Conference on Symbolic Numeric
  Computation, SNC '09, pp. 65--70. ACM, New York, NY, USA (2009)

\bibitem{kim05}
Kim, S., Kojima, M., Waki, H.: Generalized lagrangian duals and sums of squares
  relaxations of sparse polynomial optimization problems.
\newblock SIAM Journal on Optimization \textbf{15}(3), 697--719 (2005)

\bibitem{kojima05}
Kojima, M., Kim, S., Waki, H.: Sparsity in sums of squares of polynomials.
\newblock Mathematical Programming \textbf{103}(1), 45--62 (2005)

\bibitem{lasserre2001global}
Lasserre, J.B.: Global optimization with polynomials and the problem of
  moments.
\newblock SIAM Journal on Optimization \textbf{11}(3), 796--817 (2001)

\bibitem{lofberg2004yalmip}
Lofberg, J.: Yalmip: A toolbox for modeling and optimization in matlab.
\newblock In: Computer Aided Control Systems Design, 2004 IEEE International
  Symposium on, pp. 284--289. IEEE (2004)

\bibitem{lofberg2009pre}
Lofberg, J.: Pre-and post-processing sum-of-squares programs in practice.
\newblock Automatic Control, IEEE Transactions on \textbf{54}(5), 1007--1011
  (2009)

\bibitem{mats12}
Matsukawa, Y., Yoshise, A.: A primal barrier function phase i algorithm for
  nonsymmetric conic optimization problems.
\newblock Japan Journal of Industrial and Applied Mathematics \textbf{29}(3),
  499--517 (2012)

\bibitem{mont97}
Monteiro, R.: Primal--dual path-following algorithms for semidefinite
  programming.
\newblock SIAM Journal on Optimization \textbf{7}(3), 663--678 (1997)

\bibitem{motzkin1967arithmetic}
Motzkin, T.S.: The arithmetic-geometric inequality.
\newblock Inequalities (Proc. Sympos. Wright-Patterson Air Force Base, Ohio,
  1965) pp. 205--224 (1967)

\bibitem{nesterov2000squared}
Nesterov, Y.: Squared functional systems and optimization problems.
\newblock In: High performance optimization, pp. 405--440. Springer (2000)

\bibitem{nesterov1994interior}
Nesterov, Y., Nemirovskii, A., Ye, Y.: Interior-point polynomial algorithms in
  convex programming, vol.~13.
\newblock SIAM (1994)

\bibitem{PPP02}
Papachristodoulou, A., Anderson, J., Valmorbida, G., Prajna, S., Seiler, P.,
  Parrilo, P.A.: {SOSTOOLS}: Sum of squares optimization toolbox for {MATLAB}.
\newblock \texttt{http://arxiv.org/abs/1310.4716} (2013).
\newblock Available from \texttt{http://www.eng.ox.ac.uk/control/sostools},
  \texttt{http://www.cds.caltech.edu/sostools} and
  \texttt{http://www.mit.edu/\~{}parrilo/sostools}

\bibitem{Parrilo00}
Parrilo, P.A.: Structured semidefinite programs and semialgebraic geometry
  methods in robustness and optimization.
\newblock Ph.D. thesis, California Inst. of Tech. (2000)

\bibitem{Parrilo03}
Parrilo, P.A.: Semidefinite programming relaxations for semialgebraic problems.
\newblock Mathematical Programming \textbf{96}, 293--320 (2003)

\bibitem{parrilo2003minimizing}
Parrilo, P.A., Sturmfels, B.: Minimizing polynomial functions.
\newblock Algorithmic and quantitative real algebraic geometry, DIMACS Series
  in Discrete Mathematics and Theoretical Computer Science \textbf{60}, 83--99
  (2003)

\bibitem{permenter2012selecting}
Permenter, F., Parrilo, P.A.: Selecting a monomial basis for sums of squares
  programming over a quotient ring.
\newblock In: CDC, pp. 1871--1876 (2012)

\bibitem{powers1998algorithm}
Powers, V., W{\"o}rmann, T.: An algorithm for sums of squares of real
  polynomials.
\newblock Journal of pure and applied algebra \textbf{127}(1), 99--104 (1998)

\bibitem{renegar2001mathematical}
Renegar, J.: A mathematical view of interior-point methods in convex
  optimization, vol.~3.
\newblock Siam (2001)

\bibitem{renz78}
Reznick, B.: {Extremal PSD forms with few terms}.
\newblock Duke Mathematical Journal \textbf{45}, 363--374 (1978)

\bibitem{reznick1989forms}
Reznick, B.: Forms derived from the arithmetic-geometric inequality.
\newblock Mathematische Annalen \textbf{283}(3), 431--464 (1989)

\bibitem{reznick1995uniform}
Reznick, B.: Uniform denominators in hilbert's seventeenth problem.
\newblock Mathematische Zeitschrift \textbf{220}(1), 75--97 (1995)

\bibitem{sch05}
Schweighofer, M.: Optimization of polynomials on compact semialgebraic sets.
\newblock SIAM Journal on Optimization \textbf{15}(3), 805--825 (2005)

\bibitem{seiler2013sosopt}
Seiler, P.: Sosopt: A toolbox for polynomial optimization.
\newblock arXiv preprint arXiv:1308.1889  (2013)

\bibitem{seiler2013simplification}
Seiler, P., Zheng, Q., Balas, G.: Simplification methods for sum-of-squares
  programs.
\newblock arXiv preprint arXiv:1303.0714  (2013)

\bibitem{shor1998nondifferentiable}
Shor, N.Z.: Nondifferentiable optimization and polynomial problems, vol.~24.
\newblock Springer (1998)

\bibitem{Stur98}
Sturm, J.F., Zhang, S.: On the long-step path-following method for semidefinite
  programming.
\newblock Operations Research Letters \textbf{22}(4窶?), 145 -- 150 (1998)

\bibitem{sturm98}
Sturmfels, B.: Polynomial equations and convex polytopes.
\newblock The American Mathematical Monthly \textbf{105}(10), pp. 907--922
  (1998)

\bibitem{todd2001semidefinite}
Todd, M.J.: Semidefinite optimization.
\newblock Acta Numerica 2001 \textbf{10}, 515--560 (2001)

\bibitem{vand96}
Vandenberghe, L., Boyd, S.: Semidefinite programming.
\newblock SIAM Review \textbf{38}(1), 49--95 (1996)

\bibitem{waki07}
Waki, H., Kim, S., Kojima, M., Muramatsu, M.: Sums of squares and semidefinite
  program relaxations for polynomial optimization problems with structured
  sparsity.
\newblock SIAM Journal on Optimization \textbf{17}(1), 218--242 (2006)

\bibitem{waki2010facial}
Waki, H., Muramatsu, M.: A facial reduction algorithm for finding sparse sos
  representations.
\newblock Operations Research Letters \textbf{38}(5), 361--365 (2010)

\bibitem{wolkowicz2000handbook}
Wolkowicz, H., Saigal, R., Vandenberghe, L.: Handbook of semidefinite
  programming: theory, algorithms, and applications, vol.~27.
\newblock Springer (2000)

\end{thebibliography}

\end{document}